\documentclass[a4paper,reqno]{amsart}

\usepackage{amsmath} 
\usepackage{amsthm}
\usepackage{mathrsfs}
\usepackage[pdftex]{graphicx} 
\usepackage[driverfallback=pdftex]{hyperref}
\usepackage{amsfonts}
\usepackage{url}
\usepackage{lscape} 

\usepackage[all,cmtip]{xy}

\usepackage{tikz}
\usepackage{tikz-cd}
\usetikzlibrary{positioning}%
\usepackage[non-sorted-cites,initials,alphabetic,nobysame]{amsrefs}

\makeatletter
	
	\@addtoreset{equation}{section}
\makeatother

\usepackage{comment}

\usepackage{color}
\definecolor{orange}{rgb}{ .90, .60, .0}
\definecolor{light-blue}{rgb}{ .35, .70, .90}
\definecolor{bluish-green}{rgb}{ .0, .60, .50}
\definecolor{bright-yellow}{rgb}{ .95, .90, .25}
\definecolor{vermilion}{cmyk}{ 0, .80, 1.00, .0 }
\definecolor{navy-blue}{rgb}{ 0, .45, .75}
\definecolor{redish-purple}{rgb}{ .80, .60, .70}

\DeclareMathOperator{\End}{End}

\DeclareMathOperator{\Id}{Id}
\DeclareMathOperator{\Int}{Int}

\DeclareMathOperator{\Ker}{Ker}

\DeclareMathOperator{\Ima}{Im}

\DeclareMathOperator{\vol}{vol}

\DeclareMathOperator{\prim}{prim}

\newcommand{\ii}{\sqrt{-1}}
\newcommand{\adjoint}{*}
\newcommand{\dual}{*}

\newcommand{\deRham}{\mathrm{dR}}
\newcommand{\Rumin}{\mathrm{R}}
\newcommand\RN {\cE }

\newcommand{\suR}{\text{sR}}

\newcommand\Mst{\text{ s.t. }}

\newcommand\bbC{\mathbb{C}}

\newcommand\bbR{\mathbb{R}}

\newcommand\cE{\mathcal{E}}

\newcommand\cL{\mathcal{L}}

\newcommand\cQ{\mathcal{Q}}

\newcommand\sF{\mathscr{F}}

\newcommand{\wt}[1]{\widetilde{#1}}

\newcommand{\rest}[1]{\big\rvert_{#1}} 

\theoremstyle{plain}
\newtheorem{defi}{Definition}[section] 

\newtheorem{corollary}[defi]{Corollary}
\newtheorem{proposition}[defi]{Proposition}
\newtheorem*{proposition*}{Proposition}

\newtheorem{theorem}[defi]{Theorem}

\newtheorem*{theorem*}{Theorem}

\allowdisplaybreaks[1]

\begin{document}

\title[Harmonic forms and the Rumin complex]{
	Harmonic forms and the Rumin complex \\ on Sasakian manifolds
}
\author{
	Akira Kitaoka 
}
\address{Graduate School of Mathematical Sciences \\ The University of Tokyo \\ 3-8-1 Komaba, Meguro, Tokyo 153-8914 Japan}
\email{kitaoka-akira7680@ms.u-tokyo.ac.jp}

\maketitle

\begin{abstract}
  We show that the kernel of the Rumin Laplacian agrees with that of the Hodge-de Rham Laplacian on compact Sasakian manifolds.
  As a corollary, we obtain another proof of primitiveness of harmonic forms.
  Moreover, the space of harmonic forms coincides with the sub-Riemann limit of Hodge-de Rham Laplacian when its limit converges.
  Finally, we express the analytic torsion function associated with the Rumin complex in terms of the Reeb vector field.
\end{abstract}

\section{Introduction}

Let $(M,H)$ be a compact contact manifold of dimension $2n+1$,
and $E$ be a flat vector bundle associated with the unitary representation $\alpha \colon \pi_1 (M) \to \mathrm{U}(r)$.
Rumin \cite{Rumin-94} introduced a complex
$(\cE ^\bullet (M,E) , d_{\Rumin }^\bullet )$,
which is
induced by the de Rham complex $(\Omega^{\bullet}(M,E) , d )$.
Here, $\Omega^{k}(M,E)$ is the space of $E$-valued differential $k$-forms.
A specific feature of the complex is that
the operator $D = d_{\Rumin }^n \colon \cE ^{n} (M,E) \to \cE ^{n+1} (M,E) $ in `middle degree' is a second-order,
while
$d_{\Rumin }^k \colon \cE ^{k} (M,E) \to \cE ^{k+1} (M,E) $ for $k \not = n$ are first order.
Moreover, the space $\cE ^\bullet (M,E)$ is included in $\Omega^{\bullet}(M,E)$ (see \S \ref{sec: The Rumin complex}).
The complex is induced by the exterior derivatives.
Let $a_k = 1 / \sqrt{|n-k|}$ for $k \not = n$ and $a_n = 1$.
Then, $(\cE ^\bullet (M,E) , d_\RN ^\bullet)$,
where $d_\RN ^k = a_k d_{\Rumin}^k$,
is also a complex.
We call $(\cE ^\bullet (M,E) , d_\RN ^\bullet)$ the {\em Rumin complex}.

Let $\theta$ be a contact form of $H$,
and $J$ be an almost complex structure on $H$.
Let $T$ be the Reeb vector field of $\theta$.
Then we may define a Riemann metric $g_{\theta , J}$ on $TM$ by extending the Levi metric $ d \theta (- ,J-)$ on $H$ (see \S \ref{sec: The Rumin complex}).
The quadruple $(M,H,\theta,J)$ is called a Sasakian manifold if  $(M,H,J)$ is a CR manifold and $\cL_T J = 0$.
In virtue of the rescaling, $ d_\RN ^\bullet$
satisfies K\"{a}hler-type identities on Sasakian manifolds \cite{Rumin-00}.
In addition, the analytic torsion of the Rumin complex is written by using the Betti number and the Ray-Singer torsion on the lens spaces \cite{Kitaoka-20}.

Following \cite{Rumin-94}, we define the Rumin Laplacians $\Delta_\RN $ associated with \\ $(\cE ^\bullet (M,E) , d_\RN ^\bullet)$ and the metric $g_{\theta , J}$ by
\[
  \Delta_\RN ^k
  :=
  \begin{cases}
  ( d_\RN d_\RN^{\adjoint })^2
  + (d_\RN^{\adjoint } d_\RN )^2 , & k \not = n, \, n+1 , \\
  ( d_\RN d_\RN^{\adjoint })^2
  + D^{\adjoint } D , & k=n, \\
  D D^{\adjoint }
  + (d_\RN^{\adjoint } d_\RN )^2 , & k =n+1.
  \end{cases}
\]
Rumin showed that $\Delta_\RN $ has discrete eigenvalues with finite multiplicities.

\begin{theorem}
  \label{theo:harmonic_form_coincides_with_Ker_Delta_E}
  Let $(M,H,\theta,J)$ be a compact Sasakian manifold of dimension $2n+1$.
  Then, the kernel of the Rumin Laplacian agrees with that of the Hodge-de Rham Laplacian,
  that is,
  \[
    \Ker ( \Delta_{\deRham} \colon \Omega^k (M) \to \Omega^k (M) ) =
    \Ker ( \Delta_{\RN} \colon \RN^k (M) \to \RN^k (M) )
    .
  \]
\end{theorem}

Recently, Case showed that by \cite[Proposition 12.10]{Case-21},
for a compact Sasakian manifold $M$,
\begin{align}
  & \Ker ( \Delta_{\RN} \colon \RN^k (M) \to \RN^k (M) )
	\notag \\
	& =
  \bigoplus_{i+j=k}
  \Ker ( \Delta_{\RN} \colon \RN^k (M) \to \RN^k (M) ) \cap C^{\infty} \left(M, \bigwedge^{i,j} H^{\dual} \right)
  ,
  \label{eq:Hodge-decomposition}
\end{align}
where
\[
	\bigwedge^{i,j} H^{\dual}
	:=
	\bigwedge^{i}
	\left\{ \phi \in  H^{\dual} \, \middle| \, J \phi = \ii \phi \right\}
	\otimes
	\bigwedge^{j}
	\left\{ \phi \in  H^{\dual} \, \middle| \, J \phi = -\ii \phi \right\}
	.
\]
Using \eqref{eq:Hodge-decomposition},
he \cite{Case-21} recovered a topological obstruction \cite{Blair-Goldberg-67, Fujitani-66} to the existence of Sasakian structure on a given manifold in terms of its Betti numbers.

From Theorem \ref{theo:harmonic_form_coincides_with_Ker_Delta_E} and \eqref{eq:Hodge-decomposition},
we give another proof of the following corollary:
\begin{corollary}
	\label{cor:Tachi-Fuji}
  {\rm (\cite[Theorems 7.1, 8.1]{Tachibana-65}, \cite[Corollary 4.2]{Fujitani-66})}
  In the setting of Theorem \ref{theo:harmonic_form_coincides_with_Ker_Delta_E},
	for $\phi \in \Ker (\Delta_{\deRham} \colon \Omega^k (M) \to \Omega^k (M))$,
  \begin{enumerate}
    \item[(1)] if $k \leq n$, we have $\Int_T \phi=0, \, (d \theta \wedge )^{\adjoint} \phi = 0$,
    \item[(2)] if $k \geq n+1$, we have $\theta \wedge \phi=0, \, d \theta \wedge \phi = 0$,
    \item[(3)] we have $J\phi \in \Ker (\Delta_{\deRham})$, that is, $J \phi$ is also a harmonic form,
  \end{enumerate}
  where $\Int_T$ is the interior product with respect to $T$.
\end{corollary}

The Bernstein-Gelfand-Gelfand sequence $(\cE^{\bullet} (M,E) , d_{\Rumin} )$ is defined for parabolic geometry on the twisted de Rham complex due to \u{C}ap-Slov\'{a}k-Sou\u{c}ek \cite{Cap-Slovak-Soucek-01}
and Calderbank-Dimer \cite{Calderbank-Dimer-01}.
Rumin has also introduced a non $G$-invariant version in the context of sub-Riemannian geometry \cite{Rumin-99},
which coincides with the Rumin complex \cite{Rumin-94} on contact manifolds (e.g. \cite[\S 5.3]{Rumin-05}).
To the author's knowledge, the Sasakian manifolds are the only cases when the kernel of $d_{\Rumin} + d_{\Rumin}^{\adjoint}$ agrees with the harmonic space.
It is an interesting question: whether on the filtered manifolds its kernel coincides with the harmonic space or not.

Next,
one can view the Rumin complex as arising naturally the sub-Riemannian limit of $\Delta_{\deRham}$ induced by the filtration $H \subset TM$ \cite{Rumin-00}.
An analytic approach to sub-Riemannian limit,
for fiber bundles, was developed by Mazzeo and Melrose \cite{Mazzeo-Melrose-90},
and,
for Riemann foliations, was by Forman
\cite{Forman-95}.
On contact manifolds, Albin-Quan solved the asymptotical equation of $\Delta_{\deRham}$, which was introduced by Forman \cite{Forman-95},
and its asymptotic behavior can be explicitly written by the Rumin complex \cite{Albin-Quan-20}.

Let $ t \in [0, \infty )$.
The exterior algebra of $M$ splits into horizontal and vertical forms.
With respect to this decomposition, the exterior differential writes
\[
	d
	=
	d_0 + d_b + d_T,
\]
where
\[
	d_0 := \Int_T d \theta \wedge ,
	\quad
	d_T := \theta \wedge \mathcal{L}_T ,
\]
and we define $d_b$ by
$d_b \theta := 0$
and
for $\phi \in C^{\infty} \left( M, \bigwedge^{\bullet} H^{\dual} \right)$
\[
	d_b \phi :=  d \phi - \theta \wedge (\Int_T d \phi )
.
\]
We set
\[
	d_t := d_0 + t d_b + t^2 d_T .
\]

Let $X := M \times [0, \infty )$ and $\Delta_{t} := d_t d_t^{\adjoint} + d_t^{\adjoint} d_t \colon \Omega^{\bullet} (M) \to \Omega^{\bullet} (M)$ for $t \in [0, \infty )$,
where $d_t^{\adjoint}$ is the formal adjoint of $d_t$ for the $L^2$-inner product on $\Omega^{\bullet} (M)$.
We define the space of the sub-Riemannian limit differential forms by
\begin{align*}
  &
	{}^{\suR} \Omega^k (X)
  \\
	&
  :=
  \left\{
    u_0 + t u_1 + \cdots + t^q u_q \in \Omega^k (X)
  \, \middle| \,
    u_j \in \Omega^k (M)
    , \,
    q \in \mathbb{Z}_{\geq 0}
		, \,
		t \geq 0
  \right\}
  ,
\end{align*}
and set
\begin{align*}
  \sF_p^k ( \Delta_{t} )
  & :=
  \left\{
    u \in \Omega^k (M)
    \,
  \middle|
	\,
    \exists \wt u  \in {}^{\suR} \Omega^k (X)
    \Mst
    \wt u \rest{t =0} = u,
    \,
    \Delta_{t} \wt u = O (t^p)
  \right\}
  ,
\end{align*}
for $p \geq 0$.
In \cite[p. 18]{Albin-Quan-20}, Albin-Quan showed that
\begin{align*}
	\sF_3^k (\Delta_{t} ) & = \Ker ( \Delta_{\RN} \colon \RN^k (M) \to \RN^k (M) )
	& \text{ for }k \not = n , n+1,
	\\
	\sF_5^k (\Delta_{t} ) & = \Ker ( \Delta_{\RN} \colon \RN^k (M) \to \RN^k (M) )
	& \text{ for } k = n , n+1.
\end{align*}
By Corollary \ref{cor:Tachi-Fuji},
we obtain the following:
\begin{corollary}
	\label{cor:Forman_spectral_seq_on_contact}
  In the setting of Theorem \ref{theo:harmonic_form_coincides_with_Ker_Delta_E},
	\[
		\Ker (\Delta_{\deRham} ) = \bigcap_{t>0} \Ker (\Delta_{t} ) .
	\]
\end{corollary}

By Theorem \ref{theo:harmonic_form_coincides_with_Ker_Delta_E} and \cite[p. 18]{Albin-Quan-20},
we have
\begin{corollary}
	\label{cor:Forman_spectral_seq_on_contact_2}
  In the setting of Theorem \ref{theo:harmonic_form_coincides_with_Ker_Delta_E},
	\begin{align*}
		\sF_3^k (\Delta_{t} ) & = \bigcap_{t>0} \Ker (\Delta_{t} \colon \Omega^k (M) \to \Omega^k (M))
		& \text{ for }k \not = n , n+1,
		\\
		\sF_5^k (\Delta_{t} ) & = \bigcap_{t>0} \Ker (\Delta_{t} \colon \Omega^k (M) \to \Omega^k (M))
		& \text{ for } k = n , n+1.
	\end{align*}
\end{corollary}
It means that for ``$k \not = n , n+1$ and $ u \in \sF_3^k (\Delta_{t} )$''
or ``$k = n , n+1$ and $ u \in \sF_5^k (\Delta_{t} )$'', taking $\wt u = u$, we see
\[
	\Delta_{t} \wt u = 0 \text{ for } t >0.
\]
In \cite{Albin-Quan-20},
on compact contact manifolds,
for $u \in \Ker ( \Delta_{\RN} )$
Albin-Quan constructed $\wt u$ such that
``for $k \not = n , n+1$, $\Delta_t \wt u = O (s^3)$'' and ``for $k = n , n+1$, $\Delta_t \wt u = O (s^5)$''
by using $d_0, d_b ,d_T$.
On compact Sasakian manifolds,
we give a simple construction of $\wt u$.

We next introduce the analytic torsion and metric of the Rumin complex \\$(\cE ^\bullet (M,E) , d_\RN ^\bullet)$ by following \cite{Bismut-Zhang-92 , Kitaoka-19, Rumin-Seshadri-12}.
We define the {\em contact analytic torsion function} associated with $(\cE ^\bullet (M,E), d_\RN ^\bullet)$ by
\begin{equation*}
  \kappa_\RN ( M, E, g_{\theta , J}) (s)
  := \sum_{k=0}^{n} (-1)^{k+1} (n+1-k) \zeta (\Delta_\RN ^k ) (s)
  ,
\end{equation*}
where $\zeta (\Delta_\RN ^k ) (s)$ is the spectral zeta function of $\Delta_\RN ^k$
and the {\em contact analytic torsion} $ T_\RN $  by
\[
  2 \log T_\RN ( M , E, g_{\theta , J})
  =
  \kappa_\RN ( M , E, g_{\theta , J}) ' (0) .
\]
Let $H^\bullet (\cE ^\bullet (M,E), d_\RN ^\bullet)$ be the cohomology of the Rumin complex.
We define the contact metric on $\det H^\bullet (\cE ^\bullet (M,E) , d_\RN ^\bullet)$ by
\[
  \| \quad \|_\RN (M, E, g_{\theta , J} )
  =
  T_\RN^{-1} ( M , E, g_{\theta , J})
  | \quad |_{L^2 (\cE ^\bullet) },
\]
where
the metric
$| \quad |_{L^2 (\cE ^\bullet) }$ is induced by the $L^2$-metric on $\cE ^\bullet (M, E)$
via identification of the cohomology classes by the harmonic forms.

Rumin and Seshadri \cite{Rumin-Seshadri-12} defined another analytic torsion function $\kappa_{\Rumin}$ from\\
$(\cE ^\bullet (M,E) , d_{\Rumin }^\bullet )$, which is different from $\kappa_\RN $ except in dimension 3.
In dimension 3, they showed that
$\kappa_{\Rumin} (M, E, g_{\theta , J} ) (0)$ is a contact invariant,
that is, independent of the metric $g_{\theta , J}$.
Moreover, on Sasakian manifolds with $S^1$-action and dimension $3$, $\kappa_{\Rumin} (M, E, g_{\theta , J} ) (0) = 0 $.
Furthermore, for flat bundles with unimodular holonomy on Sasakian manifolds with $S^1$-action and dimension $3$, they showed that this analytic torsion and the Ray-Singer torsion $T_{\deRham } (M, E, g_{\theta , J} )$ are equal.
With this coincidence, they found a relation between the Ray-Singer torsion and holonomy.

To extend the coincidence,
with $d_{\RN}$ instead of $d_{\Rumin }$,
the author \cite{Kitaoka-20} showed that $T_{\RN } (K, E, g_{\theta , J} ) = n!^{\dim H^0 (K , E )} T_{\deRham } (K, E, g_{\theta , J} )$ on the standard lens space $K$ of dimension $2n+1$.
Moreover, Albin and Quan \cite{Albin-Quan-20} showed the difference between the Ray-Singer torsion and the contact analytic torsion is given by some integrals of universal polynomials in the local invariants of the metric
on contact manifolds.
Via any isomorphism $\Phi \colon \det H^{\bullet} (\cE^{\bullet} (M,E) , d_{\RN}^{\bullet}) \cong \det H^{\bullet} (\Omega^{\bullet} (M,E) , d) $,
\[
  \frac{\| \quad \|_{\cE}}{\| \Phi \quad \|_{\deRham}} = \frac{T_{\deRham}}{T_{\cE}} .
\]
By Theorem \ref{theo:harmonic_form_coincides_with_Ker_Delta_E}, we can choose $\Phi = \Id$ and
\[
  \frac{\| \quad \|_{\cE}}{\| \quad \|_{\deRham}} = \frac{T_{\deRham}}{T_{\cE}} .
\]

Finally, to adapt the proof of Theorem \ref{theo:harmonic_form_coincides_with_Ker_Delta_E}
to the contact analytic torsion, we express the analytic torsion function associated with the Rumin complex in terms of the Reeb vector field.
We set
\[
  2 \square_{\RN} := \sqrt{\Delta_{\RN}} + \ii \cL_T ,
  \quad
  2 \overline{\square}_{\RN} := \sqrt{\Delta_{\RN}} - \ii \cL_T .
\]
\begin{theorem}
  \label{theo:contact_torsion_Reeb_vector}
	Let $(M,H,\theta,J)$ be a compact Sasakian manifold of dimension $2n+1$,
	and $E$ be a flat vector bundle over $M$ associated with the unitary representation $\alpha \colon \pi_1 (M) \to \mathrm{U}(r)$.
	We assume that the universal covering of $M$ is compact.
	Then, we have
  \begin{align*}
    \kappa_{\RN} (s)
  & =
    \sum_{k=0}^{n} (-1)^{k+1} (n+1-k)
      \Bigl(
        \zeta ( -\mathcal{L}_T^2 |_{\Ker \square_{\RN } \cap \Ima \overline{\square}_{\RN } \cap \mathcal{E}^k (M,E) } ) (s)
  \\
  & \hspace{35mm}
        + \zeta ( -\mathcal{L}_T^2 |_{\Ima \square_{\RN } \cap \Ker \overline{\square}_{\RN} \cap \mathcal{E}^k (M,E) } ) (s)
        + \dim H^k (M ,E)
        \Bigr)
    ,
  \end{align*}
	where $\zeta$ is a spectral zeta function.
\end{theorem}
On Sasakian manifolds with $S^1$-action and dimension $3$,
Rumin-Seshadri \cite[(4.6)]{Rumin-Seshadri-12} showed the theorem above.
By using this formula,
they showed that this analytic torsion and the Ray-Singer torsion $T_{\deRham } (M, E, g_{\theta , J} )$ are equal for flat bundles with unimodular holonomy on Sasakian manifolds with $S^1$-action and dimension $3$
\cite{Rumin-Seshadri-12}.

The paper is organized as follows.
In \S \ref{sec: The Rumin complex}, we recall the definition and properties
of the Rumin complex on contact manifolds.
In \S \ref{sec: Some properties on Sasakian manifolds}, we recall some properties of the Rumin complex on Sasakian manifolds.
In \S \ref{sec:Eigenvectors of the Rumin Laplacian},
we decompose $\cE^k (M)$ as a direct sum with respect to simultaneous diagonalization of $\square_{\RN}$ and $\overline{\square}_{\RN}$, and we see the Rumin Laplacian $\Delta_{\RN}$ on them.
In \S \ref{ref:Comparison to the Hodge-de Rham Laplacian and the Rumin Laplacian}, we show the kernel of $\Delta_{\RN}$ agrees with the space of harmonic forms.
In \S \ref{sec:Forman_spectral_sequence},
we show the space of harmonic forms coincides with $\sF_\bullet^k (\Delta_{t} )$ when it converges.
In \S \ref{sec:contact torsion}, we express the analytic torsion function associated with the Rumin complex in terms of the Reeb vector field.

\subsection*{Acknowledgement}

The author is very grateful to his supervisor Professor Kengo Hirachi for introducing him to this subject and for helpful comments.
This work was supported by the program for Leading Graduate Schools, MEXT, Japan.

\section{The Rumin complex on contact manifolds}
\label{sec: The Rumin complex}

We call $(M,H)$ an orientable contact manifold of dimension $2n+1$
if $H$ is a subbundle of $TM$ of codimension $1$
and
there exists a 1-form $\theta$,
called a contact form,
such that
 $\Ker (\theta \colon TM \to \bbR ) = H $
and $\theta \wedge ( d \theta )^n \not = 0$.
The Reeb vector field of $\theta$ is the unique vector field $T$ satisfying $\theta (T) =1$ and $\Int_T d \theta =0$, where $\Int_T$ is the interior product with respect to $T$.

For $H$ and $\theta$,
we call $J \in \End (TM)$
an almost complex structure associated with $\theta$ if
$J^2 = - \Id $ on $H$, $JT=0$, and the Levi form $d \theta ( - , J - )$ is positive definite on $H$.
Given $\theta $ and $J$, we define a Riemannian metric $g_{\theta , J}$ on $TM$ by
\begin{align*}
g_{\theta , J} (X,Y) := d \theta (X , JY) + \theta (X) \theta (Y) \hspace{1cm} \text{ for } X,Y \in TM .
\end{align*}
Let $*$ be the Hodge star operator on $\bigwedge^\bullet T^*M$ with respect to $g_{\theta , J}$.

The Rumin complex \cite{Rumin-94} is defined on contact manifolds as follows.
We set $L := d \theta \wedge $
and $\Lambda := *^{-1} L *$,
which is the adjoint operator of $L$ with respect to the metric $g_{\theta , J}$ at each point.
We set
\[
  \bigwedge_{\prim}^{k} H^{\dual} :=
  \left\{
  v \in \bigwedge^{k} H^{\dual}
  \, \middle| \,
  \Lambda v = 0
  \right\},
  \quad \quad
  \bigwedge_{L}^{k} H^{\dual}
  :=
  \left\{
  v \in \bigwedge^{k} H^{\dual}
  \, \middle| \,
  L v = 0
  \right\},
\]
\[
\cE ^k (M )
:=
\left\{
\begin{aligned}
& C^\infty \left( M,\bigwedge_{\prim}^{k} H^{\dual}  \right) , \quad && k \leq n ,\\
& C^\infty \left( M, \theta \wedge \bigwedge_{L}^{k-1} H^{\dual}  \right) , && k \geq n+1 .
\end{aligned}
\right.
\]
We embed $H^{\dual}$ into $T^{\dual}M$ as the subbundle $ \{ \phi \in T^* M \mid \phi (T) = 0 \}$ so that
we can regard
\begin{align*}
\Omega_H^{k} (M) := & \, C^{\infty} \left( M, \bigwedge^{k} H^{\dual}  \right)
\end{align*}
as a subspace of $\Omega^k (M)$, the space of $k$-forms.
We define $d_b \colon \Omega_H^{k} (M) \to \Omega_H^{k+1} (M)$ by
\[
d_b \phi := \, d \phi - \theta \wedge (\Int_T d \phi )
,
\]
and then $D \colon \cE ^n (M) \to \cE ^{n+1} (M) $ by
\begin{equation}
  D = \theta \wedge ( \mathcal{L}_T + d_b L^{-1} d_b ),
  \label{eq:D-0}
\end{equation}
where $\mathcal{L}_T$ is the Lie derivative with respect to $T$,
and we use the fact that $L \colon \bigwedge^{n-1} H^{\dual} \\ \to\bigwedge^{n+1} H^{\dual}$ is an isomorphism.

Let $P \colon \bigwedge^{k} H^{\dual} \to \bigwedge_{\prim}^{k} H^{\dual} $ be the fiberwise orthogonal projection with respect to $g_{\theta , J}$, which also defines a projection $P \colon \Omega^k (M) \to \cE ^k (M) $.
We set
\[
d_{\Rumin }^k
:=
\begin{cases}
P \circ d & \text{ on } \cE ^k (M), \quad k \leq n-1, \\
D & \text{ on } \cE ^n (M) , \\
d & \text{ on } \cE ^k (M) , \quad k \geq n+1.
\end{cases}
\]
Then $(\cE ^\bullet (M), d_{\Rumin }^\bullet)$ is a complex.
Let $d_\RN ^{k} = a_k d_{\Rumin }^{k}$, where
$a_k = 1 / \sqrt{|n-k|}$ for $k \not = n$ and
$a_n = 1$.
We call $(\cE ^\bullet (M), d_\RN ^\bullet)$ the Rumin complex.

\begin{proposition}{\rm (\cite[p. 286--287]{Rumin-94})}
\label{prop:BGG_coh_agrees_with_deRahm_coh}

Let $M$ be an orientable contact manifold.
Then, the Rumin complex forms a fine resolution of the constant sheaf $\mathbb{R}$.
Hence its cohomology coincides with the de Rham cohomology of $M$.
\end{proposition}

We define the $L^2$-inner product on $\Omega^{k} (M)$ by
\[
(\phi , \psi ) := \int_M g_{\theta , J} (\phi , \psi ) d \vol_{g_{\theta , J}}
\]
and the $L^2$-norm on $\Omega^k (M)$ by $\| \phi \| := \sqrt{(\phi , \phi )}$.
Since the Hodge star operator $*$ induces a bundle isomorphism
from $ \bigwedge_{\prim}^{k} H^{\dual}$ 
to $\theta \wedge \bigwedge_{L}^{2n-k} H^{\dual}$,
it also induces a map $\cE ^{k} (M ) \to \cE ^{2n+1-k} (M) $.

Let $d_\RN ^{\adjoint }$ and $D^{\adjoint }$ denote the formal adjoint of $d_\RN $ and $D$, respectively for the $L^2$-inner product.
We define the fourth-order Laplacians $\Delta_\RN $ on $\cE ^k (M )$ by
\[
\Delta_\RN ^k
 :=
\begin{cases}
( d_\RN ^{k-1} d_\RN ^{k-1} {}^{\adjoint })^2 + (d_\RN ^{k} {}^{\adjoint } d_\RN ^{k} )^2 , & k \not = n , n+1 , \\
\vspace{1mm}
( d_\RN ^{n-1} d_\RN ^{n-1} {}^{\adjoint })^2 + D^{\adjoint } D , & k=n, \\
D D^{\adjoint } + (d_\RN ^{n+1} {}^{\adjoint } d_\RN ^{n+1} )^2 , & k =n+1.
\end{cases}
\]
We call it the Rumin Laplacian \cite{Rumin-94}.

We recall the Rumin Laplacian $\Delta_{\RN}$ has Rockland condition or hypoellipticity (e.g. \cite[\S 3]{Rumin-94}, \cite[Proposition 3.5.8]{Ponge-08}).
For example, \cite[Propositions 5.5.2, 5.5.9]{Ponge-08} follows the following proposition:
\begin{proposition}
	\label{eq:sub-elliptic-Rumin_Laplacian}
	Let $M$ be a compact orientable contact manifold of dimension $2n+1$.
	Then,
	\begin{enumerate}
		\item[(1)]
		the Rumin Laplacian $\Delta_{\RN}$ has discrete eigenvalues,
		\item[(2)]
		the kernel of $\Delta_{\RN}$ is isomorphic to the cohomology of the Rumin complex as vector space, that is,
		\[
		\Ker (\Delta_{\RN}) \cong H^{\bullet} (\RN^{\bullet}(M) , d_{\RN}).
		\]
	\end{enumerate}
\end{proposition}

\section{The Rumin complex on the Sasakian manifolds}
\label{sec: Some properties on Sasakian manifolds}

For a contact manifold $(M, H )$ and an almost complex structure $J$, we decompose the bundles defined in the previous subsection as follows:
\begin{align*}
H^{1,0} &:= \{ v \in \bbC H  \mid Jv = \sqrt{-1} v \} , &
H^{0,1}& := \{ v \in \bbC H \mid Jv = - \sqrt{-1} v \}, \\
\bigwedge^{i, j} H^{\dual} & :=  \bigwedge^{i} H^{\dual 1,0} \otimes \bigwedge^{j} H^{\dual 0,1} , &
\bigwedge_{\prim}^{i, j} H^{\dual} & :=  \left\{ \phi \in \bigwedge^{i, j} H^{\dual} \, \middle| \, \Lambda \phi =0 \right\} , \\
\Omega_H^{i, j}& := C^{\infty} \left( M, \bigwedge^{i, j} H^{\dual} \right) . &
%
%
\end{align*}
We call $(M, H, \theta , J )$ a {\em Sasakian manifold}
if $\mathcal{L}_T J =0$ and
\begin{align*}
		[C^{\infty} (M, H^{1,0}) , C^{\infty} (M, H^{1,0})]
	\subset
		C^{\infty} (M, H^{1,0}).
\end{align*}
Then $d_b \Omega_H^{i , j} \subset \Omega_H^{i+1 , j} \oplus \Omega_H^{i , j+1} $.
We define $\partial_b \colon \Omega_H^{i , j}$
$\to$
$\Omega_H^{i+1 , j} $ and $\overline{\partial}_b \colon \Omega_H^{i , j} \to \Omega_H^{i , j+1 } $ by
\[
d_b = \partial_b + \overline{\partial}_b.
\]
We set
\begin{align*}
	\partial_{\Rumin }
	& :=
	\begin{cases}
		P \partial_b , &
		k \leq n-1 ,
		\\
		\partial_b , &
		k \geq n,
	\end{cases}
	&
	\overline{\partial}_{\Rumin }
	& :=
	\begin{cases}
		P \overline{\partial}_b , &
		k \leq n-1,
		\\
		\overline{\partial}_b , &
		k \geq n,
	\end{cases}
	\\
	\partial_{\RN }
	& :=
		a_k \partial_{\Rumin }
		,
	&
	\overline{\partial}_{\RN }
	& :=
		a_k \overline{\partial}_{\Rumin }
		.
\end{align*}
Similarly, we decompose on $\cE^k (M)$ for $k < n $
\[
d_{\Rumin} = \partial_{\Rumin} + \overline{\partial}_{\Rumin}
,
\quad
d_\RN  = \partial_\RN  + \overline{\partial}_\RN
.
\]
Since
\begin{align}
    \partial_b^{\adjoint }
  =
    \sqrt{-1} [ \Lambda , \overline{\partial}_b ] ,
    \,
    \overline{\partial}_b^{\adjoint }
  =
    - \sqrt{-1} [ \Lambda , \partial_b ]
    ,
    \,
    \partial_b
  =
    \sqrt{-1} [ L , \overline{\partial}_b^{\adjoint } ] ,
  \,
    \overline{\partial}_b
  =
    - \sqrt{-1} [ L , \partial_b^{\adjoint } ]
    ,
  \label{eq:adjoint_formula_of_del_b_and_del_ber_b}
\end{align}
we may rewrite \eqref{eq:D-0} as
\begin{equation}
D = \theta \wedge \left( \mathcal{L}_T - \sqrt{-1} (\partial_b + \overline{\partial}_b )
( \partial_b^{\adjoint } - \overline{\partial}_b^{\adjoint } )\right)
.
\label{eq:D-1}
\end{equation}
We set
\[
  \Delta_{\overline{\partial}_{b }} := \overline{\partial}_{b } \overline{\partial}_{b }^{\adjoint } + \overline{\partial}_{b }^{\adjoint } \overline{\partial}_{b }
  , \quad
  \Delta_{\partial_{b }} := {\partial_{b }} \partial_{b }^{\adjoint } + \partial_{b }^{\adjoint } {\partial_{b }}
  .
\]
We note that
\begin{equation}
  %
  [\partial_b , \overline{\partial}_b^{\adjoint } ] = 0 ,
  \quad
  [\overline{\partial}_b , \partial_b^{\adjoint } ] = 0 .
  \label{prop:adjoint_formula_of_del_b_and_del_ber_b_2}
\end{equation}

\section{Eigenvectors of the Rumin Laplacian}
\label{sec:Eigenvectors of the Rumin Laplacian}

Henceforth,
we assume that $(M,H,\theta,J)$ is a compact Sasakian manifold of dimension $2n+1$.
Since $*$ and $\Delta_\RN $ commute,
to determine the eigenvalues on $\cE ^\bullet (M ) $,
it is enough to compute them on $\cE ^{k } (M)$ for $k \leq n$.
We set
\[
  \Delta_{\overline{\partial}_{\RN }} := \overline{\partial}_{\RN } \overline{\partial}_{\RN }^{\adjoint } + \overline{\partial}_{\RN }^{\adjoint } \overline{\partial}_{\RN }
  , \quad
  \Delta_{\partial_{\RN }} := {\partial_{\RN }} \partial_{\RN }^{\adjoint } + \partial_{\RN }^{\adjoint } {\partial_{\RN }}
  .
\]
We recall the differential operators $\partial_{\RN }$
and $\overline{\partial}_{\RN }$ have a property as 	\eqref{prop:adjoint_formula_of_del_b_and_del_ber_b_2}, i.e., as the following proposition:
\begin{proposition}
  {\rm (\cite[(34)]{Rumin-00})}
	\label{prop:commuteness_holomorphic_Rumin_derivative}

  \begin{enumerate}
    \item[(1)]
    	On $\mathcal{E}^k(M)$, for $k \leq n$, we have
      \[
      	[\partial_{\RN }, \overline{\partial}_{\RN }^{\adjoint } ] =0 ,
        \quad
      	[\overline{\partial}_{\RN }, \partial_{\RN }^{\adjoint } ] =0 .
      \]
    \item[(2)]
      For $k \leq n-1$,
      we have
      \[
      \sqrt{\Delta_{\RN }}
      =
      \Delta_{\overline{\partial}_{\RN }} + \Delta_{\partial_{\RN }}
      .
      \]
    \item[(3)]
    For $k \leq n-1$, we have
    \[
      \sqrt{-1} \mathcal{L}_T
      =
       \Delta_{\overline{\partial}_{\RN }} - \Delta_{\partial_{\RN }} .
    \]
  \end{enumerate}
\end{proposition}
From Proposition \ref{prop:commuteness_holomorphic_Rumin_derivative} (1), we have the following proposition:
\begin{proposition}
  \label{prop:formulas-holomorphic-Rumin-Laplacian-2}

  For $k \leq n-1$, the operators $\Delta_{\partial_{\RN }}$ and $\Delta_{\overline{\partial}_{\RN }}$ commute.
\end{proposition}

We set
\begin{align*}
	\cQ^k (\lambda_{10}, \lambda_{01})
	:=
	\begin{cases}
		\left\{
		\phi \in \mathcal{E}^k (M)
		\middle|
		\Delta_{\partial_{\RN }} \phi = \lambda_{10} \phi,
		\Delta_{\overline{\partial}_{\RN }} \phi = \lambda_{01} \phi
		\right\} ,
		& k \leq n-1 ,
		\\
		\partial_{\RN } \cQ^{n-1} (\lambda_{10}, \lambda_{01})
		+ \overline{\partial}_{\RN } \cQ^{n-1} (\lambda_{10}, \lambda_{01})
		,
		& k=n.
	\end{cases}
\end{align*}
From Propositions \ref{eq:sub-elliptic-Rumin_Laplacian} (1),
\ref{prop:commuteness_holomorphic_Rumin_derivative} (2),
and \ref{prop:formulas-holomorphic-Rumin-Laplacian-2},
we decompose $\cE^k (M)$ into a direct sum of $\cQ^k (\lambda_{10}, \lambda_{01})$ for $k \leq n-1$, that is,
\begin{equation}
  \cE^k ( M ) = \bigoplus_{\lambda_{10} ,\lambda_{01} } \cQ^k (\lambda_{10}, \lambda_{01})
  .
  \label{eq:decomposition_cE^k}
\end{equation}
It follows that for $k \leq n$,
\begin{align*}
		\mathcal{E}^k (M)
	& =
		( \Ker \Delta_{\partial_{\RN }} \cap \Ker \Delta_{\overline{\partial}_{\RN }} )
		\oplus ( \Ker \Delta_{\partial_{\RN }} \cap \Ima \Delta_{\overline{\partial}_{\RN }} )
	\\
	& \quad
		\oplus ( \Ima \Delta_{\partial_{\RN }} \cap \Ker \Delta_{\overline{\partial}_{\RN }} )
		\oplus ( \Ima \Delta_{\partial_{\RN }} \cap \Ima \Delta_{\overline{\partial}_{\RN }} ) .
\end{align*}
We research each component.

About $\Ker \Delta_{\partial_{\RN }} \cap \Ker \Delta_{\overline{\partial}_{\RN }}$, from Proposition \ref{prop:commuteness_holomorphic_Rumin_derivative} (2), we have
\begin{align*}
	\Ker \Delta_{\partial_{\RN }} \cap \Ker \Delta_{\overline{\partial}_{\RN }} = \Ker \Delta_{\RN }.
\end{align*}

Next, we consider the space $\Ima \Delta_{\partial_{\RN }} \cap \Ima \Delta_{\overline{\partial}_{\RN }}$.

\begin{proposition}
	\label{prop:decomposition_of_simultaneous_diagonalization_of_box_and_box_bar}

  For $k \leq n-1$ and $\lambda_{10} , \lambda_{01} >0$, $\cQ^k (\lambda_{10}, \lambda_{01})$ has the following decomposition:
	\begin{align*}
		\cQ^k (\lambda_{10}, \lambda_{01})
	& =
		\cQ^k (\lambda_{10}, \lambda_{01}) \cap \Ima \partial_{\RN }^{\adjoint } \cap \Ima \overline{\partial}_{\RN }^{\adjoint }
		\oplus \cQ^k (\lambda_{10}, \lambda_{01}) \cap \Ima {\partial_{\RN }} \cap \Ima \overline{\partial}_{\RN }^{\adjoint }
	\notag \\
	& \quad
		\oplus \cQ^k (\lambda_{10}, \lambda_{01} ) \cap \Ima \partial_{\RN }^{\adjoint } \cap \Ima {\overline{\partial}_{\RN }}
		\oplus \cQ^k (\lambda_{10}, \lambda_{01} ) \cap \Ima {\partial_{\RN }} \cap \Ima {\overline{\partial}_{\RN }}.
	\end{align*}
\end{proposition}

\begin{proof}
	In the same way as the Hodge-de Rham Laplacian, we have
	\begin{align*}
			\Ima \Delta_{\partial_{\RN }}
		=
			\Ima {\partial_{\RN }}
			\oplus \Ima \partial_{\RN }^{\adjoint },
		\hspace{10mm}
			\Ima \Delta_{\overline{\partial}_{\RN }}
		=
			\Ima {\overline{\partial}_{\RN }}
			\oplus \Ima \overline{\partial}_{\RN }^{\adjoint }.
	\end{align*}
	Therefore, we get
	\begin{align}
			\Ima \Delta_{\partial_{\RN }} \cap \Ima \Delta_{\overline{\partial}_{\RN }}
		& =
			\Ima \partial_{\RN }^{\adjoint } \cap \Ima \overline{\partial}_{\RN }^{\adjoint }
			\oplus \Ima {\partial_{\RN }} \cap \Ima \overline{\partial}_{\RN }^{\adjoint }
		\notag \\
		& \quad
			\oplus \Ima \partial_{\RN }^{\adjoint } \cap \Ima {\overline{\partial}_{\RN }}
			\oplus \Ima {\partial_{\RN }} \cap \Ima {\overline{\partial}_{\RN }}
		\notag.
	\end{align}
	Thus, it is clear that we check this proposition.
\end{proof}

From Propositions \ref{prop:commuteness_holomorphic_Rumin_derivative} (1)
and \ref{prop:decomposition_of_simultaneous_diagonalization_of_box_and_box_bar},
we have the following proposition:
\begin{proposition}
	\label{prop:isomorphim_among_eigenspaces_of_the_Rumin_Laplacian_Im_Im}

  We assume that $\lambda_{10} , \lambda_{01} > 0$.
  Then, we have
	\begin{enumerate}
		\item[{\rm (1)}]
		for $k \leq n-2$,
		the following operators are isomorphisms:
		\begin{figure}[ht]\begin{center}
			\begin{tikzcd}
				\cQ^{k} (\lambda_{10}, \lambda_{01} ) \cap \Ima \partial_{\RN }^{\adjoint } \cap \Ima \overline{\partial}_{\RN }^{\adjoint }
					\arrow{d}[swap]{\partial_{\RN } }
					\arrow{r}{\overline{\partial}_{\RN } }
				&
				\cQ^{k+1} (\lambda_{10}, \lambda_{01} ) \cap \Ima \partial_{\RN }^{\adjoint } \cap \Ima {\overline{\partial}_{\RN }}
					\arrow{d}[swap]{\partial_{\RN } }
				\\
				\cQ^{k+1} (\lambda_{10}, \lambda_{01} ) \cap \Ima {\partial_{\RN }} \cap \Ima \overline{\partial}_{\RN }^{\adjoint }
				\arrow{r}{\overline{\partial}_{\RN } }
				&
				\cQ^{k+2} (\lambda_{10}, \lambda_{01} ) \cap \Ima {\partial_{\RN }} \cap \Ima {\overline{\partial}_{\RN }}
				,
			\end{tikzcd}
		\end{center}\end{figure}
		\item[{\rm (2)}]
    for $k = n-1$,
		the following operators are isomorphisms:
		\begin{figure}[ht]\begin{center}
			\begin{tikzcd}
				\cQ^{n-1} (\lambda_{10}, \lambda_{01} ) \cap \Ima \partial_{\RN }^{\adjoint } \cap \Ima \overline{\partial}_{\RN }^{\adjoint }
					\arrow{d}[swap]{\partial_{\RN } }
					\arrow{r}{\overline{\partial}_{\RN } }
				&
				\cQ^{n} (\lambda_{10}, \lambda_{01} ) \cap \Ima \partial_{\RN }^{\adjoint } \cap \Ima {\overline{\partial}_{\RN }}
				\\
				\cQ^{n} (\lambda_{10}, \lambda_{01} ) \cap \Ima {\partial_{\RN }} \cap \Ima \overline{\partial}_{\RN }^{\adjoint }
				,
				&
			\end{tikzcd}
		\end{center}\end{figure}
  	\item[{\rm (3)}]
    for $k \leq n-1$, the following operators are isomorphisms:
    \begin{figure}[ht]\begin{center}
      \begin{tikzcd}
        \cQ^{k} (0, \lambda_{01} ) \cap \Ima \overline{\partial}_{\RN }^{\adjoint }
          \arrow{r}{\overline{\partial}_{\RN } }
        &
        \cQ^{k+1} (0, \lambda_{01} ) \cap \Ima {\overline{\partial}_{\RN }},
        \\
        \cQ^{k} (\lambda_{10} , 0 ) \cap \Ima \partial_{\RN }^{\adjoint }
          \arrow{r}{\partial_{\RN } }
        &
        \cQ^{k+1} (\lambda_{10} , 0 ) \cap \Ima {\partial_{\RN }} .
      \end{tikzcd}
    \end{center}\end{figure}
	\end{enumerate}
\end{proposition}

Next, we consider the parts $\Ker \Delta_{\partial_{\RN }} \cap \Ima \Delta_{\overline{\partial}_{\RN }}$ and $\Ima \Delta_{\partial_{\RN }} \cap \Ker \Delta_{\overline{\partial}_{\RN }}$.
\begin{proposition}
	\label{prop:isomorphim_among_eigenspaces_of_the_Rumin_Laplacian_Ker_Im}
	On the following spaces
	\begin{enumerate}
		\item [(1)]
		$\mathcal{E}^k (M) \cap \Ker \Delta_{\partial_{\RN }} \cap \Ima \Delta_{\overline{\partial}_{\RN }}$ for $k \leq n-1$,
		\item [(2)]
		$\mathcal{E}^n (M) \cap \Ima \overline{\partial}_{\RN } \cap \Ker \Delta_{\partial_{\RN }}$,
		\item [(3)]
		$\mathcal{E}^k (M) \cap \Ima \Delta_{\partial_{\RN }} \cap \Ker \Delta_{\overline{\partial}_{\RN }}$ for $k \leq n-1$, and
		\item [(4)]
		$\mathcal{E}^n (M) \cap \Ima \partial_{\RN } \cap \Ker \Delta_{\overline{\partial}_{\RN }}$,
	\end{enumerate}
	we have
	\[
		\Delta_{\RN }
		=
		- \mathcal{L}_T^2 .
	\]
\end{proposition}

\begin{proof}
	We consider $\mathcal{E}^k (M) \cap \Ker \Delta_{\partial_{\RN }} \cap \Ima \Delta_{\overline{\partial}_{\RN }}$ for $k \leq n-1$.
	By Proposition \ref{prop:commuteness_holomorphic_Rumin_derivative} (3),
	we have
	\begin{equation*}
		\sqrt{-1} \mathcal{L}_T
	=
		 \Delta_{\overline{\partial}_{\RN }}
	= \sqrt{\Delta_{\RN }}
	,
	\end{equation*}
	that is,
	\[
		\Delta_{\RN }
		=
		- \mathcal{L}_T^2
		.
	\]
	From Propositions \ref{prop:commuteness_holomorphic_Rumin_derivative} (1)
	and \ref{prop:decomposition_of_simultaneous_diagonalization_of_box_and_box_bar},
	we obtain (2).
	Similarly,
	we get (3) and (4).
\end{proof}

  We consider the space $\Ker \partial_{\RN }^{\adjoint } \cap \Ker \overline{\partial}_{\RN }^{\adjoint } \cap \cE^n (M)$.
  From \eqref{eq:D-1},
  for $\phi \in \Ker \partial_{\RN }^{\adjoint } \cap \Ker \overline{\partial}_{\RN }^{\adjoint } \cap \cE^n (M)$,
  we have
  \begin{equation}
    \Delta_{\RN} \phi = D^{\adjoint } D \phi = - \mathcal{L}_T^2 \phi
		.
    \label{eq:isomorphim_among_eigenspaces_of_the_Rumin_Laplacian_Ker_Ker}
  \end{equation}
  We set for $\nu \in \mathbb{R}$
  \[
  \cQ^n (\nu )
  :=
  \left\{
    \phi \in \Ker \partial_{\RN }^{\adjoint } \cap \Ker \overline{\partial}_{\RN }^{\adjoint } \cap \cE^n (M)
  \, \middle| \,
    \cL_T \phi = \ii \nu \phi
  \right\}
  \subset
  \{
    \Delta_{\RN} \phi = \nu^2 \phi
  \}
  .
  \]
  We decompose $\Ker \partial_{\RN }^{\adjoint } \cap \Ker \overline{\partial}_{\RN }^{\adjoint } \cap \cE^n (M)$ into a direct sum,
  \[
    \Ker \partial_{\RN }^{\adjoint } \cap \Ker \overline{\partial}_{\RN }^{\adjoint } \cap \cE^n (M)
    =
    \bigoplus_{\nu }
    \cQ^n (\nu )
    .
  \]

  We consider $\Ima \partial_{\RN } + \Ima \overline{\partial}_{\RN }$.
  From \eqref{eq:decomposition_cE^k}, we have
  \begin{align}
  		\left( \Ima \partial_{\RN } + \Ima \overline{\partial}_{\RN } \right) \cap \cE^n (M)
  	& =
  		\partial_{\RN }
  			\bigoplus \cQ^{n-1} (\lambda_{10 } , \lambda_{01} )
  		+ \overline{\partial}_{\RN }
  			\bigoplus \cQ^{n-1} (\lambda_{10 } , \lambda_{01} )
  	\notag \\
  	& =
  				\bigoplus
					\left(
					\partial_{\RN } \cQ^{n-1} (\lambda_{10 } , \lambda_{01} )
  			+
  				\overline{\partial}_{\RN } \cQ^{n-1} (\lambda_{10 } , \lambda_{01} )
					\right)
          .
    \notag
  \end{align}

	To calculate all eigenvalues of $\Delta_{\RN}$ on $\cE^n (M)$,
  from
	Propositions
	\ref{prop:isomorphim_among_eigenspaces_of_the_Rumin_Laplacian_Im_Im},
	and
	\ref{prop:isomorphim_among_eigenspaces_of_the_Rumin_Laplacian_Ker_Im}
  it is enough to consider the space
  \[
  W := \cQ^{n-1} (\lambda_{10}, \lambda_{01}) \cap \Ima \partial_{\RN }^{\adjoint } \cap \Ima \overline{\partial}_{\RN }^{\adjoint }
	.
  \]
Let $\psi \in W \setminus \{ 0 \}$.
We set
\[
  \underline{\psi}^{(0,0)} = \psi / \| \psi \|
  ,
  \quad
  \underline{\psi}^{(1,0)} = \partial_{\RN } \psi / \| \partial_{\RN } \psi \|
  ,
  \quad
  \underline{\psi}^{(0,1)} = \overline{\partial}_{\RN } \psi / \| \overline{\partial}_{\RN } \psi \|
  .
\]
Then, we have
\[
  d_{\RN } \underline{\psi}^{(0,0)}
  =
  \sqrt{\lambda_{10}}
  \underline{\psi}^{(1,0)}
  +
  \sqrt{\lambda_{01}}
  \underline{\psi}^{(0,1)}
\]
and
\[
d_{\RN } d_{\RN }^{\adjoint }
\left(
  \sqrt{\lambda_{10}}
  \underline{\psi}^{(1,0)}
  +
  \sqrt{\lambda_{01}}
  \underline{\psi}^{(0,1)}
\right)
=
(\lambda_{10} + \lambda_{01} )
\left(
  \sqrt{\lambda_{10}}
  \underline{\psi}^{(1,0)}
  +
  \sqrt{\lambda_{01}}
  \underline{\psi}^{(0,1)}
\right)
.
\]
Therefore, we have the eigenvalue of $\Delta_{\RN}$ on $d_{\RN} W$:
\begin{align}
	(\lambda_{10} + \lambda_{01} )^2.
\label{eq:isomorphim_among_eigenspaces_of_the_Rumin_Laplacian_Im_Im_3}
\end{align}
Let us find an eigenvalue on $(d_{\RN} W)^{\bot }$, which is the orthogonal complement of $\partial_{\RN } W
+
\overline{\partial}_{\RN } W$.
We note that
\[
  \sqrt{\lambda_{01}}
  \underline{\psi}^{(1,0)}
  -
  \sqrt{\lambda_{10}}
  \underline{\psi}^{(0,1)}
\in (d_{\RN} W)^{\bot }.
\]
Let $\lambda_{T}$ be the eigenvalue of $- \ii \cL_T$ on $\cQ^{n}(\lambda_{10},\lambda_{01} )$,
 $A = \lambda_{T} - 2 \lambda_{10} $ and $B = \lambda_{T} + 2 \lambda_{01} $.
By \eqref{eq:D-1}, we see
\begin{align*}
D
\left(
  \sqrt{\lambda_{01}}
  \underline{\psi}^{(1,0)}
  -
  \sqrt{\lambda_{10}}
  \underline{\psi}^{(0,1)}
\right)
& = \sqrt{-1} \theta \wedge
\left(
  A \sqrt{\lambda_{01}}
  \underline{\psi}^{(1,0)}
  - B \sqrt{\lambda_{10}}
  \underline{\psi}^{(0,1)}
\right)
.
\end{align*}
Since $D \left(
  \sqrt{\lambda_{10}}
  \underline{\psi}^{(1,0)}
  +
  \sqrt{\lambda_{01}}
  \underline{\psi}^{(0,1)}
\right) = 0$, we have
\begin{align*}
D^{\adjoint} D
\left(
  \sqrt{\lambda_{01}}
  \underline{\psi}^{(1,0)}
  -
  \sqrt{\lambda_{10}}
  \underline{\psi}^{(0,1)}
\right)
=
\frac{A^2 \lambda_{01} + B^2 \lambda_{10}}{\lambda_{10} + \lambda_{01}}
\left(
  \sqrt{\lambda_{01}}
  \underline{\psi}^{(1,0)}
  -
  \sqrt{\lambda_{10}}
  \underline{\psi}^{(0,1)}
\right)
.
\end{align*}
We note that
\[
	\frac{|A|^2 \lambda_{01} + |B|^2 \lambda_{10} }{\lambda_{10} + \lambda_{01} }
	= \frac{(\lambda_{T} - 2 \lambda_{10} )^2 \lambda_{01} + ( \lambda_{T} + 2\lambda_{01} )^2 \lambda_{10} }{\lambda_{10} + \lambda_{01} }
	= \lambda_{T}^2 + 4 \lambda_{10} \lambda_{01} .
\]
From Proposition \ref{prop:commuteness_holomorphic_Rumin_derivative} (3), we see
\begin{align*}
	\lambda_{T}^2 + 4 \lambda_{10} \lambda_{01} = (\lambda_{10} - \lambda_{01} )^2 + 4 \lambda_{10} \lambda_{01} = (\lambda_{10} + \lambda_{01} )^2 .
\end{align*}
We get the eigenvalue on $(d_{\RN} W)^{\bot }$
\begin{align*}
	(\lambda_{10} + \lambda_{01} )^2.
\end{align*}
Therefore, we obtain the eigenvalue on $(d_{\RN} W)^{\bot }$:
the eigenvalue of $\Delta_{\RN }$ is
\begin{equation}
	(\lambda_{10} + \lambda_{01} )^2.
	\label{eq:isomorphim_among_eigenspaces_of_the_Rumin_Laplacian_Im_Im_4}
\end{equation}

We obtain the following proposition:
\begin{proposition}
  \label{prop:eigen_delta_RN_is_positive}
  On $\Ima \partial_{\RN } + \Ima \overline{\partial}_{\RN }$, the operator $\Delta_{\RN}$ is positive.
\end{proposition}

\begin{proof}

  From Propositions \ref{prop:commuteness_holomorphic_Rumin_derivative} (1) and
  \ref{prop:isomorphim_among_eigenspaces_of_the_Rumin_Laplacian_Im_Im},
  \eqref{eq:isomorphim_among_eigenspaces_of_the_Rumin_Laplacian_Im_Im_3} and
  \eqref{eq:isomorphim_among_eigenspaces_of_the_Rumin_Laplacian_Im_Im_4},
  on $\cQ^{k} (\lambda_{10}, \lambda_{01} )$ for $k \leq n$, the eigenvalue is
  \[
    (\lambda_{10} + \lambda_{01} )^2 .
  \]
  If $\lambda_{10} >0$ or $\lambda_{01} >0$, the operator $\Delta_{\RN}$ is positive.

\end{proof}

\section{Comparison beween the Hodge-de Rham Laplacina and the Rumin Laplacian}
\label{ref:Comparison to the Hodge-de Rham Laplacian and the Rumin Laplacian}

The exterior algebra of $M$ splits into horizontal and vertical forms, which we will denote by
\begin{align*}
  \Omega^{\bullet} (M) = \Omega_H^{\bullet} (M) \oplus \theta \wedge \Omega_H^{\bullet} (M) .
\end{align*}
With respect to this decomposition, the exterior differential writes
\begin{align*}
    d (\alpha_H + \theta \wedge \alpha_T )
  =
    (d_b \alpha_H + L \alpha_T ) + \theta \wedge (\mathcal{L}_T \alpha_H - d_b \alpha_T ) ,
\end{align*}
that is,
\begin{align*}
    d
  =
    \begin{pmatrix}
      d_b & L \\
      \mathcal{L}_T & - d_b
    \end{pmatrix}
		.
\end{align*}
The Hodge-de Rham Laplacian satisfies
\[
    \Delta_{\deRham}
  =
    \begin{pmatrix}
      \Delta_b + \mathcal{L}_T^{\adjoint} \mathcal{L}_T + L \Lambda &&
      [d_b^{\adjoint}, L] + [d_b , \mathcal{L}_T^{\adjoint}]
      \\
      [\Lambda , d_b] + [ \mathcal{L}_T , d_b^{\adjoint}] &&
      \Delta_b + \mathcal{L}_T \mathcal{L}_T^{\adjoint} + \Lambda L
    \end{pmatrix}
    ,
\]
where
\[
  \Delta_b = d_b d_b^{\adjoint} + d_b^{\adjoint} d_b .
\]
On Sasakian manifolds, from \eqref{eq:adjoint_formula_of_del_b_and_del_ber_b}, we see
\begin{equation}
	\Delta_{\deRham}
  =
	\begin{pmatrix}
		\Delta_b + \mathcal{L}_T^{\adjoint} \mathcal{L}_T + L \Lambda
		&&
		\sqrt{-1} \partial_b - \sqrt{-1} \overline{\partial}_b
		\\
		- \sqrt{-1} \partial_b^{\adjoint} + \sqrt{-1} \overline{\partial}_b^{\adjoint}
		&&
		\Delta_b + \mathcal{L}_T \mathcal{L}_T^{\adjoint} + \Lambda L
	\end{pmatrix}
  .
	\label{eq:deRham_and_contact_structure_on_Sasakian}
\end{equation}

Let $\phi \in \Ker (\Delta_{\RN}) \cap \cE^{k} (M)$ for $k \leq n-1$.
Since $d_b^{\adjoint} $ and $\Lambda$ commute, we act $d_b^{\adjoint}$ to $\phi$,
\[
  d_b^{\adjoint} \phi = d_{\Rumin}^{\adjoint} \phi = 0 .
\]
From Proposition \ref{prop:commuteness_holomorphic_Rumin_derivative} (1),
$\Delta_{\partial_{\RN}}\phi = \Delta_{\overline{\partial}_{\RN}}\phi = 0$.
We act $\Lambda d_b$ to $\phi$, from \eqref{eq:adjoint_formula_of_del_b_and_del_ber_b},
\[
  \Lambda d_b \phi
  = [\Lambda , d_b ] \phi + d_b \Lambda \phi
  = \ii (- \partial_b^{\adjoint} + \overline{\partial}_b^{\adjoint} ) \phi
  = 0 .
\]
It follows that
\[
  d_b \phi
  =
  d_{\Rumin} \phi
  =
  0
  .
\]
Therefore, we obtain
\begin{equation}
  \Delta_b \phi = 0.
  \label{eq:act_delta_b_to_Harmonic}
\end{equation}
From Propositions
\ref{prop:commuteness_holomorphic_Rumin_derivative} (1) and (3),
we have
\begin{equation}
  \cL_T \phi = - \Delta_{\partial_{\RN }} + \Delta_{\overline{\partial}_{\RN }} \phi = 0.
  \label{eq:act_cL_T_to_Harmonic}
\end{equation}
From
\eqref{eq:deRham_and_contact_structure_on_Sasakian},
\eqref{eq:act_delta_b_to_Harmonic}, \eqref{eq:act_cL_T_to_Harmonic} and Proposition \ref{prop:commuteness_holomorphic_Rumin_derivative} (1),
we get
\begin{align*}
  \Delta_{\deRham} \phi
  = \left( \Delta_b + \mathcal{L}_T^{\adjoint} \mathcal{L}_T + L \Lambda \right) \phi
  + \theta \wedge \left( \sqrt{-1} \partial_b - \sqrt{-1} \overline{\partial}_b \right)
    \phi
  = 0.
\end{align*}
It means that for $k \leq n-1$,
\[
	\Ker ( \Delta_{\deRham} \colon \Omega^k (M) \to \Omega^k (M) ) \subset
	\Ker ( \Delta_{\RN} \colon \RN^k (M) \to \RN^k (M) )
  .
\]
We note that $\Ker (\Delta_{\deRham })$ is a finite dimensional vector space since $M$ is compact.
From Propositions \ref{prop:BGG_coh_agrees_with_deRahm_coh} and \ref{eq:sub-elliptic-Rumin_Laplacian} (2), for $k \leq n-1$, we conclude
\[
	\Ker ( \Delta_{\deRham} \colon \Omega^k (M) \to \Omega^k (M) ) =
	\Ker ( \Delta_{\RN} \colon \RN^k (M) \to \RN^k (M) )
  .
\]

Next we consider the case $k = n$.
Let $\phi \in \Ker (\Delta_{\RN}) \cap \cE^{n} (M)$.
From Proposition \ref{prop:eigen_delta_RN_is_positive},
we see
\begin{equation}
  \phi \in \Ker \partial_{\RN }^{\adjoint } \cap \Ker \overline{\partial}_{\RN }^{\adjoint }.
  \label{eq:kernel_RN_condition}
\end{equation}
From \eqref{eq:kernel_RN_condition} and \eqref{eq:isomorphim_among_eigenspaces_of_the_Rumin_Laplacian_Ker_Ker}, we have
\begin{equation}
  \cL_T \phi = 0.
  \label{eq:act_cL_T_to_Harmonic-2}
\end{equation}
From \eqref{eq:kernel_RN_condition}, we see
\begin{equation}
  \partial_b^{\adjoint} \phi = \partial_{\Rumin }^{\adjoint } \phi = 0,
  \quad
  \overline{\partial}_b^{\adjoint} \phi = \overline{\partial}_{\Rumin }^{\adjoint } \phi = 0.
  \label{eq:act_del_b_adj_to_Harmonic-2}
\end{equation}
Since $\cE^n (M) = \Ker (L) $, from \eqref{eq:adjoint_formula_of_del_b_and_del_ber_b},
we have
\begin{equation*}
  \partial_b \phi = \ii [ L , \overline{\partial}_b^{\adjoint} ] \phi = 0 ,
  \quad
  \overline{\partial}_b \phi = \ii [ L , \partial_b^{\adjoint} ] \phi = 0 .
\end{equation*}
We obtain
\begin{equation}
  \Delta_b \phi = 0.
  \label{eq:act_delta_b_to_Harmonic-2}
\end{equation}
From \eqref{eq:deRham_and_contact_structure_on_Sasakian},
\eqref{eq:act_cL_T_to_Harmonic-2},
\eqref{eq:act_del_b_adj_to_Harmonic-2}
and \eqref{eq:act_delta_b_to_Harmonic-2},
we conclude
\[
  \Delta_{\deRham} \phi = 0.
\]
It means that
\[
  \Ker ( \Delta_{\RN} \colon \RN^n (M) \to \RN^n (M) )
	\subset
	\Ker ( \Delta_{\deRham} \colon \Omega^n (M) \to \Omega^n (M) )
  .
\]
In the same way as the case $k \leq n-1$,
from Propositions \ref{prop:BGG_coh_agrees_with_deRahm_coh} and \ref{eq:sub-elliptic-Rumin_Laplacian} (2), we conclude
\[
  \Ker ( \Delta_{\RN} \colon \RN^n (M) \to \RN^n (M) ) =
	\Ker ( \Delta_{\deRham} \colon \Omega^n (M) \to \Omega^n (M) )
  .
\]

\section{Proof of Corollary \ref{cor:Forman_spectral_seq_on_contact}}
\label{sec:Forman_spectral_sequence}

From Corollary \ref{cor:Tachi-Fuji}, for $\phi \in \Ker ( \Delta_{\deRham} \colon \Omega^k (M) \to \Omega^k (M) )$, we get
\[
	d \phi = d_0 \phi + d_b \phi + d_T \phi = 0 ,
	\quad
	d_0 \phi =0,
	\quad
	d^{\adjoint} \phi = d_0^{\adjoint} \phi + d_b^{\adjoint} \phi + d_T^{\adjoint} \phi = 0,
	\quad
	d_0^{\adjoint} \phi = 0
	.
\]
Since $d_b \phi $ and $ d_T \phi $ are linearly independent, and $d_b^{\adjoint} \phi $ and $ d_T^{\adjoint} \phi$ are also linearly independent,
\[
	d_0 \phi = d_b \phi = d_T \phi
	= d_0^{\adjoint} \phi = d_b^{\adjoint} \phi = d_T^{\adjoint} \phi = 0 .
\]
Therefore, for $t \geq 0$, we obtain
\[
	\Delta_{t} \phi = (d_t d_t^{\adjoint} + d_t^{\adjoint} d_t ) \phi = 0 .
\]
Since
\[
\Ker (\Delta_{\deRham} ) \supset \bigcap_{t>0} \Ker (\Delta_{t} ),
\]
we have
\[
\Ker (\Delta_{\deRham} ) = \bigcap_{t>0} \Ker (\Delta_{t} ),
\]
that is,
we conclude Corollary \ref{cor:Forman_spectral_seq_on_contact}.

\section{Proof of Theorem \ref{theo:contact_torsion_Reeb_vector}}
\label{sec:contact torsion}

We set
\[
  2 \square_{\RN} := \sqrt{\Delta_{\RN}} + \ii \cL_T ,
  \quad
  2 \overline{\square}_{\RN} := \sqrt{\Delta_{\RN}} - \ii \cL_T .
\]
From \eqref{eq:isomorphim_among_eigenspaces_of_the_Rumin_Laplacian_Ker_Ker} and
Propositions \ref{prop:isomorphim_among_eigenspaces_of_the_Rumin_Laplacian_Im_Im}
and \ref{prop:isomorphim_among_eigenspaces_of_the_Rumin_Laplacian_Ker_Im},
we have
\begin{align*}
		\kappa_{\RN} (s)
	& =
		\sum_{k=0}^n (-1)^{k+1} (n+1-k) \zeta (\Delta_{\RN , k} ) (s)
	\\
	& =
		\sum_{k=0}^{n} (-1)^{k+1} (n+1-k)
			\Bigl(
				\zeta ( -\mathcal{L}_T^2 |_{\Ker \square_{\RN } \cap \Ima \overline{\square}_{\RN } \cap \mathcal{E}^k (M,E) } ) (s)
	\\
	& \hspace{40mm}
				+ \zeta ( -\mathcal{L}_T^2 |_{\Ima \square_{\RN } \cap \Ker \overline{\square}_{\RN} \cap \mathcal{E}^k (M,E) } ) (s)
  \\
	& \hspace{40mm}
        + \dim H^k (M,E)
        \Bigr)
    .
\end{align*}

\end{document}